\documentclass[10pt]{amsart}
\usepackage{amstext,amsfonts,amssymb,amscd,amsbsy,amsmath,verbatim}
\usepackage{ifthen}
\usepackage{amsthm}
\usepackage{latexsym}
\usepackage[all]{xy}
\usepackage{enumerate}

\newlength{\lengthHold}

\newcommand{\bbZ}{\ensuremath{\mathbb{Z}}}
\newcommand{\bbN}{\ensuremath{\mathbb{N}}}

\newcommand{\hilb}[2]{\ensuremath{{#2}(t)}}
\newcommand{\tens}[1]{\ensuremath{\mathbb{T}\left(#1\right)}}

\newcommand{\pd}{\ensuremath{\operatorname{pd}}}
\newcommand{\gldim}{\ensuremath{\operatorname{gldim}}}

\newcommand{\syz}{\ensuremath{\Omega}}

\newcommand{\eps}{\ensuremath{\epsilon}}
\newcommand{\m}{\ensuremath{\mathfrak{m}}}
\newcommand{\p}{\ensuremath{\mathfrak{p}}}
\newcommand{\q}{\ensuremath{\mathfrak{q}}}
\newcommand{\Ker}{\ensuremath{\operatorname{Ker}}}

\newcommand{\Ext}{\ensuremath{\operatorname{Ext}}}

\newcommand{\Hom}{\ensuremath{\operatorname{Hom}}}
\newcommand{\depth}{\ensuremath{\operatorname{depth}}}

\newcommand{\rank}{\ensuremath{\operatorname{rank}}}
\newcommand{\til}[1]{\ensuremath{\widetilde{#1}}}
\newcommand{\shift}{\ensuremath{\Sigma}}

\newcommand{\dual}[2]{\ensuremath{\xi_{#1}^{#2}}}
\newcommand{\freemultword}[3]{\ensuremath{ {}^{#1}[{\mathsf{#2}}#3]}}
\newcommand{\freemultpower}[3]{\ensuremath{ {}^{#1}[{\mathsf{#2}}^{#3}]}}
\newcommand{\freemult}[2]{\ensuremath{{}^{#1}\mathsf {[#2]}}}
\newcommand{\free}[2]{\ensuremath{{}^{#1}\mathsf {#2}}}
\newcommand{\bas}{\ensuremath{\mathsf}}

\newcommand{\del}{\ensuremath{\partial}}

\newcommand{\calE}{\ensuremath{\mathcal{E}}}
\newcommand{\calF}{\ensuremath{\mathcal{F}}}
\newcommand{\calL}{\ensuremath{\mathcal{L}}}
\newcommand{\calM}{\ensuremath{\mathcal{M}}}
\newcommand{\calN}{\ensuremath{\mathcal{N}}}
\newcommand{\calR}{\ensuremath{\mathcal{R}}}
\newcommand{\calS}{\ensuremath{\mathcal{S}}}
\newcommand{\calT}{\ensuremath{\mathcal{T}}}
\newcommand{\calX}{\ensuremath{\mathcal{X}}}
\newcommand{\extRalg}{\ensuremath{\calR}}
\newcommand{\extSalg}{\ensuremath{\calS}}
\newcommand{\extTalg}{\ensuremath{\calT}}
\newcommand{\extMR}{\ensuremath{\calM_R}}
\newcommand{\extMS}{\ensuremath{\calM_S}}
\newcommand{\extNR}{\ensuremath{\calN_R}}
\newcommand{\extNT}{\ensuremath{\calN_T}}
\newcommand{\extLR}{\ensuremath{\calL}}

\newcommand{\amal}{\ensuremath{\sqcup}}

\newcommand{\scs}[1]{\ensuremath{\scriptstyle{#1}}}

\newcommand{\matrSetup}{\renewcommand{\arraystretch}{.5}
                        \setlength{\lengthHold}{\arraycolsep}
                        \setlength{\arraycolsep}{2pt}}
\newcommand{\matrReset}{\setlength{\arraycolsep}{\lengthHold}
                        \renewcommand{\arraystretch}{1}}

\newcommand{\onto}{\ensuremath{\twoheadrightarrow}}

\renewcommand{\to}{\ensuremath{\rightarrow}}
\newcommand{\from}{\ensuremath{\leftarrow}}

\newcommand{\lto}{\ensuremath{\longrightarrow}}
\newcommand{\xra}{\ensuremath{\xrightarrow}}
\newcommand{\xla}{\ensuremath{\xleftarrow}}

\SelectTips{eu}{}

\theoremstyle{plain}
\newtheorem{introThm}{Theorem}

\newtheorem{introCor}[introThm]{Corollary}

\newtheorem{theorem}{Theorem}[section]
\newtheorem{proposition}[theorem]{Proposition}
\newtheorem{corollary}[theorem]{Corollary}
\newtheorem{lemma}[theorem]{Lemma}

\theoremstyle{definition}
\newtheorem{definition}[theorem]{Definition}
\newtheorem{subsec}[theorem]{}

\newtheorem{example}[theorem]{Example}
\newtheorem{notation}[theorem]{Notation}
\newtheorem{construction}[theorem]{Construction}

\theoremstyle{remark}
\newtheorem{remark}[theorem]{Remark}

\numberwithin{equation}{section}

\begin{document}

\title[Cohomology over fiber products]{Cohomology over Fiber Products of Local Rings}
\author{W. Frank Moore}
\date{\today}
\thanks{This research was partially supported through NSF grant \#DMS-0201904.}

\begin{abstract}
Let $S$ and $T$ be local rings with common residue field $k$, let $R$ be the
fiber product $S \times_k T$, and let $M$ be an $S$-module.  The Poincar\'e series $P^R_M$ of $M$
has been expressed in terms of $P^S_M$, $P^S_k$ and $P^T_k$ by Kostrikin and Shafarevich, and by Dress and Kr\"amer.
Here, an explicit minimal resolution, as well as theorems on the structure of $\Ext_R(k,k)$ and $\Ext_R(M,k)$
are given that illuminate these equalities.  Structure theorems for the cohomology modules of fiber products of modules are also given.
As an application of these results, we compute the depth of cohomology modules over a fiber product.
\end{abstract}

\maketitle

\section*{Introduction} \label{secIntro}

In homological investigations one often has information on properties of a module over a certain ring,
and wants to extract information on its properties over a different ring.  In this paper we consider the following
situation: $S \to k \from T$ are surjective homomorphisms of rings, $k$ is a field, $R$ is the fiber product $S \times_k T$,
and $M$ an $S$-module.  We further assume that $S$ and $T$ are either local rings with common residue field
$k$, or connected graded $k$-algebras.

The starting point of this paper is the construction of an explicit minimal free resolution of $M$, viewed
as an $R$-module, from minimal resolutions of $M$ and $k$ over $S$ and $k$ over $T$.  This is carried out
in section \ref{secComplexConstr}.  The structure of the $R$-free resolution allows us to obtain precise information
on the multiplicative structure of cohomology over $R$.   Some of the results obtained in this work have
been proved in the graded case by use of standard resolutions.  However, no similar approach can be used in
the local case.

The symbol $\amal$ denotes a coproduct, also known as a free product, of $k$-algebras.  

\begin{introThm} \label{introExtIso}
The canonical homomorphism of graded $k$-algebras
\begin{equation*}\Ext_S(k,k) \amal \Ext_T(k,k) \to \Ext_{S \times_k T}(k,k)\end{equation*}
defined by the universal property of coproducts of $k$-algebras
is bijective.  For every $S$-module $M$, the canonical homomorphism of graded left $\Ext_R(k,k)$-modules
\begin{equation*}\Ext_R(k,k) \otimes_{\Ext_S(k,k)} \Ext_S(M,k) \to \Ext_R(M,k)\end{equation*}
defined by the multiplication map, is bijective.
\end{introThm}

These isomorphisms relate the Poincar\'e series $P^R_M(t)$ of $M$ over
$R$ to $P^S_M(t)$, $P^S_k(t)$ and $P^T_k(t)$; this relationship was proved for $M = k$ by
Kostrikin and Shafarevich \cite{KosSha57}, and by Dress and Kr\"amer \cite[Theorem 1]{DreKra75} in the present setting.
In \cite{PolPos05}, Polishchuk and Positselski proved the preceding theorem, when $S$ and $T$ are
connected $k$-algebras by using cobar constructions.

By combining Theorem \ref{introExtIso} with an observation of Dress and Kr\"amer concerning
second syzygy modules over fiber products, we obtain the following corollary.

\begin{introCor}
Let $L$ be an $R$-module.  One then has $\Omega^2 L = M \oplus N$ where $M$ and $N$ are $S$ and $T$-modules respectively,
and an exact sequence of graded left $\extRalg$-modules
\begin{equation*}0 \to \left(\shift^{-2}\extRalg \otimes_{\extSalg} \Ext_S(M,k)\right)\oplus\left(\shift^{-2}\extRalg \otimes_{\extTalg} \Ext_T(N,k)\right)
    \to \extLR \to \extLR/\extLR^{\geq 2} \to 0\end{equation*}
where $\extRalg = \Ext_R(k,k)$, $\extSalg = \Ext_S(k,k)$, $\extTalg = \Ext_T(k,k)$ and $\extLR = \Ext_R(L,k)$.  
\end{introCor}

Theorem \ref{introExtIso} shows that $\Ext_{-}(k,k)$, as a functor in the ring argument, transforms products into coproducts.
We show that $\Ext_R(-,k)$, as a functor from $R$-modules to $\Ext_R(k,k)$-modules, has a similar property.
In the graded setting, this was shown by Polishchuk and Positselski \cite{PolPos05}.  Their methods do not extend
to the local case, where even the equality of Poincar\'e series given is new.
\begin{introThm} \label{introFibMod}
Let $M$, $N$, and $V$ be $S$, $T$ and $k$-modules respectively, so that we may define
the $R$-module $M \times_V N$ as in Theorem \ref{thmExtFibMod}.
The exact sequence of $R$-modules
\begin{equation*}0 \to M \times_V N \xra{\iota} M \times N \xra{\mu-\nu} V \to 0\end{equation*}
induces an exact sequence of graded left $\Ext_R(k,k)$-modules
\matrSetup
\begin{equation*}0 \to \Ext_R(V,k) \xra{\left(\begin{array}{c} \scs{\mu^*} \\ \scs{-\nu^*} \end{array}\right)} 
        \Ext_R(M,k)\oplus\Ext_R(N,k) \xra{\iota^*} \Ext_R(M \times_V N,k) \to 0.\end{equation*}
\matrReset
\end{introThm}

In section \ref{secDepth} we study the depth of $\Ext_R(M,k)$ over $\Ext_R(k,k)$ for an $R$-module $M$.
The notion of depth was used in \cite{FHJLT88} to study the homotopy Lie algebras of simply connected
CW complexes, and of local rings.  More recently, Avramov and Veliche \cite{AvrVel05} have shown that small depth of
$\Ext_R(M,k)$ over $\Ext_R(k,k)$ is responsible for significant complications in the structure
of the stable cohomology of $M$ over $R$. For cohomology modules of $R$-modules,
large depth is impossible. 

\begin{introThm}
Let $M$ be an $R$-module.  Then
\begin{equation*}\depth_{\Ext_R(k,k)} \Ext_R(M,k) \leq 1,\end{equation*}
with equality if $M$ is an $S$-module.  In particular, one has $\depth \Ext_R(k,k) = 1$.
\end{introThm}

\section{Graded Hilbert series and modules} \label{gradedAlgebras}
In this section, we set notation for the entire article.
Let $k$ be a field.

\begin{subsec}
A $k$-algebra $A$ is \emph{graded} if there is a decomposition of
$A = \bigoplus_{i \in \bbZ} A_i$ as $k$-vector spaces, and for all $i,j \in \bbZ$, one has
$A_i A_j \subseteq A_{i+j}$.  We use both upper and lower indexed graded
objects and adopt the notation $A^i = A_{-i}$.  One says that $A$ is \emph{connected} if $A_0 = K$ and $A_i = 0$ for
$i < 0$ (or equivalently, $A^i = 0$ for $i > 0$).
\end{subsec}
\begin{subsec} \label{gradedModule}
A left module $M$ over a graded $k$-algebra $A$ is \emph{graded} if there is a decomposition
$M = \bigoplus_{i \in \bbZ}M_i$ as $k$-vector spaces, and for all $i,j \in \bbZ$, one has
$A_i M_j \subseteq M_{i+j}$.  
\end{subsec}

\begin{subsec}
Let $M$ be a graded free $K$-module such that $M_i = 0$ for $i \ll 0$ and $\dim_k M_i$ is finite for all $i$.
The Hilbert series of $M$ is the formal Laurent series
\begin{equation*}\hilb{A}{M} = \sum_i \rank_k M_i t^i.\end{equation*}
If $\hilb{A}{M}$ is defined, we say that $M$ has a Hilbert series.  Let $\tens{M}$ denote the tensor
algebra of $M$ over $A$.
\end{subsec}
\begin{lemma} \label{hilbTensor} Let $M$ and $N$ be graded free $A$-modules with Hilbert series.  Then one has an equality
\begin{equation*}\hilb{A}{M \otimes_A N} = \hilb{A}{M}\hilb{A}{N}.\end{equation*}
If $M_i = 0$ for $i \leq 0$, then one also has
\begin{equation*}\hilb{A}{\tens{M}} = \dfrac{1}{1 - \hilb{A}{M}}\end{equation*}
\end{lemma}
\begin{proof}
The first claim is clear. For the second, since $M_i = 0$ for $i \leq 0$, $\tens{M}$ has a Hilbert series.
Furthermore, there is an isomorphism as graded $A$-modules $\tens{M} \cong A \oplus (M \otimes_A \tens{M})$.
The desired result follows.
\end{proof}

\begin{subsec} \label{freeProdAlg}
If $A$ and $B$ are graded connected $k$-algebras, the coproduct of $A$ and $B$ in this category is
the free product of $A$ and $B$, denoted $A \amal B$,  and can be described as follows:  A $k$-basis for $(A \amal B)_n$
consists of all elements of the form
\begin{equation*}
\begin{gathered}
a_1 \otimes b_2 \otimes \cdots \otimes a_l, \quad \phantom{\text{and}} \quad a_1 \otimes b_2 \otimes \cdots \otimes b_m,\\
b_1 \otimes a_2 \otimes \cdots \otimes b_p \quad \text{and} \quad b_1 \otimes a_2 \otimes \cdots \otimes a_q
\end{gathered}
\end{equation*}
where $a_i$ and $b_j$ range over homogeneous bases of $A$ and $B$ respectively, and the total degree of an elementary
tensor(given by summing the degrees of its terms) is $n$.  Multiplication in $A \amal B$ is given by
\begin{equation*}(v \otimes \cdots \otimes x)(y \otimes \cdots \otimes w) = \begin{cases}
                          v \otimes \cdots \otimes xy \otimes \cdots \otimes w &
                              \text{for $x,y \in A$ or $x,y \in B$} \\
                          v \otimes \cdots \otimes x \otimes y \otimes \cdots \otimes w &
                              \text{otherwise.}\end{cases}\end{equation*}
We say that an element of the form $a \otimes \cdots \otimes b$
with $a \in A$ and $b \in B$ starts with $A$ and ends with $B$.  
\end{subsec}

\begin{subsec} \label{freeProdMod}
Let $M$ be a graded left $A$-module, then $(A \sqcup B) \otimes_A M$ is a graded left $A \sqcup B$-module.
A $k$-basis for $((A \sqcup B) \otimes_A M)_n$ consists of all elements of the form 
\begin{equation*}a_1 \otimes b_2 \otimes \cdots \otimes b_p \otimes m_{p+1} \quad\text{and} \quad b_1 \otimes a_2 \otimes \cdots \otimes b_q \otimes m_{q+1}\end{equation*}
where $a_i$, $b_j$, and $m_k$ range over homogeneous bases of $A$, $B$ and $M$ respectively, and the total degree of an elementary tensor
is $n$.  The action of $A \sqcup B$ is given by
\begin{equation*}(v \otimes \cdots \otimes x)(y \otimes \cdots \otimes w) = \begin{cases}
                          v \otimes \cdots \otimes xy \otimes \cdots \otimes w &
                             \text{for $x,y \in A$ or $x,y \in B$} \\
                          v \otimes \cdots \otimes xw &
                             \text{for $x \in A$ and $y \in P$} \\
                          v \otimes \cdots \otimes x \otimes y \otimes \cdots \otimes w &
                             \text{otherwise.}\end{cases}\end{equation*}
\end{subsec}
\begin{lemma} \label{hilbLemma}
Let $A$ and $B$ be graded conected $k$-algebras and $M$ a graded left $A$-module.  There is an equality of Hilbert series
\begin{equation*}\hilb{k}{\left((A \sqcup B) \otimes_A M\right)} = \dfrac{\hilb{k}{M}\hilb{k}{B}}{\hilb{k}{A} + \hilb{k}{B} - \hilb{k}{A} \hilb{k}{B}}.\end{equation*}
\end{lemma}
\begin{proof}
The basis given in \ref{freeProdAlg} shows there is an isomorphism of $k$-vector spaces
\[A \sqcup B \cong B \otimes_k \tens{A_+ \otimes_k B_+} \otimes_k A.\]
Tensoring over $A$ with $M$ on the right gives
\begin{equation*}
(A \sqcup B)\otimes_A M  \cong \left(B \otimes_k \tens{A_+ \otimes_k B_+}\right)\otimes_k M.
\end{equation*}
A computation of Hilbert series gives the desired equality.
\end{proof}

\section{Resolutions over a fiber product} \label{secComplexConstr}

We consider a diagram of homomorphisms of rings
\begin{equation}
\xymatrix{S \times_{k} T \ar[r]^-{\tau} \ar[d]_{\sigma} & T \ar[d]^{\pi_T} \\
            S \ar[r]_{\pi_S} & k}\end{equation}
where $\pi_S$ and $\pi_T$ are surjective, and $S \times_k T$ is the fiber product:
\begin{equation*}S \times_k T = \{(s,t) \in S \times T~\colon~\pi_S(s) = \pi_T(t)\}.\end{equation*}
We set $\p = \Ker \pi_S$, $\q = \Ker \pi_T$, and $\m = \Ker \pi_S\sigma = \Ker \pi_T\tau$.  One
then has $\m = \p \oplus \q$ and we identify $\p$ and $\q$ with subsets of $R$.
Every $S$-module is considered an $R$-module via $\sigma$, and similarly for $T$-modules.  

\begin{subsec} \label{ringSettings}
In the sequel we assume that we are in one of the following situations:
\begin{itemize}
\item $S$ and $T$ are commutative, noetherian, local rings with common residue field $k$, $\p$ and $\q$ the
  maximal ideals of $S$ and $T$ respectively, and $M$ is a finitely generated $S$-module; or
\item $S$ and $T$ are non-negatively graded, connected, degree-wise finite $k$-algebras, $\p = S_+$, $\q = T_+$,
  and $M$ is a graded, bounded below, degree-wise finite $S$-module.
\end{itemize}
\end{subsec}

For a ring $A$ as in \ref{ringSettings}, we say that a complex of free $A$-modules $X$ is \emph{minimal} if it satisfies
$\del(X) \subseteq \m X$, where $\m$ is the (homogeneous) maximal ideal of $A$.  Note that every $S$-module $M$
we consider has a minimal free resolution in which each free module is finitely generated.

If $L$ is an $A$-module, the Poincar\'e series of $L$ over $A$ is the formal power series
\begin{equation*}P^A_L(t) := \sum_i \dim_k \Ext_A^i(L,k)\,t^i.\end{equation*}
Thus the coefficient of $t^i$ is the rank of the $i^\text{th}$ free module in a minimal resolution of $L$ over $A$, when one exists.

\begin{definition} \label{basedDefn}
Let $A$ be a ring and $\bas X$ a graded set, $\bas X = \bigsqcup_{n \geq 0} \bas X_n$.  
We let $\free AX$ denote the graded free left $A$-module with basis $\bas X_n$ in degree $n$, and set
$\free AX_n = 0$ when $\bas X_n = \emptyset$.  We call $\free AX$ a \emph{graded based module} over $A$ with basis $\bas X$.
Homomorphisms of based modules are identified with their matrices in the chosen bases.

For a based module $\free AX$, we identify $A \otimes_A \free AX$ and $\free AX$ by means of the canonical isomorphism.
We use $\freemult A{XY}$ to denote the graded based $A$-module
$\free AX \otimes_A \free AY$ with graded basis $\bas {XY} = \bigsqcup_n [{\bas{XY}}]_n$, where 
$[{\bas {XY}}]_n$ is the set of symbols
\begin{equation*}\{xy \mid x \in {\bas X}_i, y \in {\bas Y}_j, i+j = n\}.\end{equation*}
\end{definition}

\begin{construction} \label{complConstr}
Let $M$ be an $S$-module.  Let $P \to M$ and $E \to k$ be free resolutions of $M$, respectively $k$, over
$S$, and let $T \to k$ be a free resolution of $k$ over $T$ such that $E_0 = S$ and $F_0 = T$.
Choose bases $\bas P$, $\bas E$, and $\bas F$ of the graded modules $P$, $E$, and $F$ over $S$, $S$ and $T$, respectively,
so that $\bas E_0 = \{1_S\}$ and $\bas F_0 = \{1_T\}$.
Consider the elements of $\bas P$, $\bas E_{\geq 1}$ and $\bas F_{\geq 1}$ as letters of an alphabet.
The degree of a word in this alphabet is defined to be the sum of the degrees of its letters.

Let $\bas G$ be the set of all words of the form
\begin{equation*}\{f_1e_2f_3\cdots e_{2l-2}f_{2l-1}p_{2l}\}\quad\text{and}\quad\{e_1f_2e_3\cdots e_{2l-1}f_{2l}p_{2l+1}\}\end{equation*}
where $e_i,f_i$ and $p_i$ range over $\bas E_{\geq 1}, \bas F_{\geq 1}$ and $\bas P$ respectively, and $l \geq 0$.
Form the free graded $R$-module $G = \free RG$.

Every word $w \in \bas G$ has the form $xw'$ for some letter $x$ and a (possibly empty) word $w'$.
Assume that one has $\del(E) \subseteq \p E$, $\del(P) \subseteq \p P$, $\del(F) \subseteq \q F$, and set
\begin{equation*}\del^{G}(w) = \begin{cases}
            \del^{P}(x)   & \text{for}~x \in \bas P \\
            \del^{E}(x)w' & \text{for}~x \in \bas E \\
            \del^{F}(x)w' & \text{for}~x \in \bas F, \\
            \end{cases}\end{equation*}
and extend $\del^G$ to a endomorphism of $G$ by $R$-linearity.  
Set $\del^G_i = \del|_{G_i}$.
We remark that a matrix $\varphi$ with entries in $\p$ defines a homomorphism $^{S}\varphi$ of free $S$-modules,
as well as a homomorphism $^{R}\varphi$ of free $R$-modules. 
\end{construction}

\begin{remark} \label{gDiffDiagram}
The first few degrees of the complex $G$ in Construction \ref{complConstr} looks as follows:
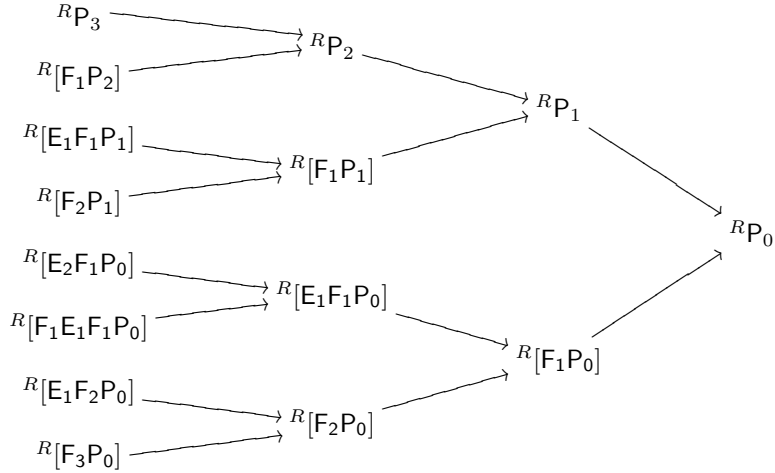
\begin{equation*}\xymatrix@R=-.4pc@C=3.5pc{
\free RP_3 \ar@<+.7ex>[dr] & & & \\
 & \free RP_2 \ar[ddr] & & \\
\freemult R{F_1P_2} \ar[ur] & & & \\
 & & \free RP_1 \ar[ddddr] & \\
\freemult R{E_1F_1P_1} \ar[dr] & & \\
 & \freemult R{F_1P_1} \ar[uur] & & \\
\freemult R{F_2P_1} \ar[ur] & & & \\
 & & & \free RP_0 \\
\freemult R{E_2F_1P_0} \ar[dr] & & & \\
 & \freemult R{E_1F_1P_0} \ar[ddr] & & \\
\freemult R{F_1E_1F_1P_0} \ar[ur] & & & \\
 & & \freemult R{F_1P_0} \ar[uuuur] & \\
\freemult R{E_1F_2P_0} \ar[dr] & & & \\
 & \freemult R{F_2P_0} \ar[uur] & & \\
\freemult R{F_3P_0} \ar@<-.5ex>[ur] & & &
}\end{equation*}      
Note that each map in the diagram acts on the leftmost letter of a word.  \end{remark}

\begin{remark} \label{gradedRemark}
Assume that we are in the graded situation of \ref{ringSettings}.
If $P$, $E$ and $F$ are complexes of graded $S$, $S$ and $T$-modules respectively,
then $G$ is a complex of graded $R$-modules; the
internal degree of a word is the sum of the internal degrees of the letters in the word.
\end{remark}

\begin{theorem} \label{reslnThm}
Assume that:
\begin{itemize}
\item $M$ is an $S$-module with a minimal free resolution $P$,  
\item $S/\p$ has a minimal free resolution $E$,
\item $T/\q$ has a minimal free resolution $F$.
\end{itemize}
The maps of free modules $\del^G_i$ defined in Construction \ref{complConstr}
\begin{equation*}G \colon \cdots \to G_i \xra{\del^G_i} G_{i-1} \to \cdots \to G_1 \xra{\del^G_1} G_0 \lto 0,\end{equation*}
give a free resolution of the $R$-module $M$ and satisfies $\del^G(G) \subseteq \m(G)$.
\end{theorem}

The following corollary was first obtained in \cite{KosSha57} for $M = k$ and in \cite[Theorem 1]{DreKra75} in general.

\begin{corollary} \label{DreKraResult}
There is an equality of Poincar\'e series:
\begin{equation*}P^R_M(t) = \dfrac{P^S_M(t)P^T_k(t)}{P^S_k(t) + P^T_k(t) - P^S_k(t)P^T_k(t)}.\end{equation*}
\end{corollary}
\begin{proof}
One may describe the basis $\bas G$ in Construction \ref{complConstr} as
a basis of $k$-vector space $\free kF \otimes_R \tens{\free kE_{\geq 1}\otimes_k\free kF_{\geq 1}} \otimes_k \free kP$.
Therefore, Lemma \ref{hilbTensor} gives
\begin{equation*}\hilb{R}{G} = \frac{\hilb{R}{P}\hilb{R}{F}}{1 - (\hilb{R}{E} - 1)(\hilb{R}{F} - 1)}.\end{equation*}
The resolutions used in Construction \ref{complConstr} are minimial, so one has
$\hilb{R}{G} = P^R_M(t)$, $\hilb{R}{E} = P^S_k(t)$, etc.  Thus, the formula above gives the desired equality.
\end{proof}

\begin{proof}[Proof of Theorem \ref{reslnThm}] To show that $G$ is a complex,
let $w$ be a word of degree $i$, with $i \geq 2$.  Suppose $w = xyw'$ where $w'$
is a word, $x$ is a letter of degree 1 and $y$ is an arbitrary letter.  For $x \in \bas E$ and $y \in \bas F$ one has
\begin{equation*}\del^2(w) = \del(\del^E(x)yw') \in \del(\p yw') = \p\del(yw') = \p\del^F(y)w' \subseteq \p\q w' = 0.\end{equation*}
The cases with $x \in \bas F_1$ and $y \in \bas P$, and with $x \in \bas F_1$ and $y \in \bas E$ are similar.
If $w = xw'$ where $x$ is a letter of degree greater than or equal to 2, then $\del^2(w) = 0$, since
$\free RP, \free RE$, and $\free RF$ are complexes of $R$-modules, and hence one has $\del^2(x) = 0$.

Let $\freemult R{E_{\geq 2}G}$ denote the $R$-linear span of words whose first letter is in $\bas E_i$ for
some $i \geq 2$, see Definition \ref{basedDefn}.  Let $\freemult R{F_1E_{\geq 1}G}$ denote the span of words
starting with a letter from $\bas F_1$, followed by a letter from $\bas E$.  Symbols such as
$\freemult R{F_{\geq 2}G}, \freemult R{E_1F_{\geq 1}G}$, etc\. are defined similarly.

Let $a$ be an element of $G_i$, with $i \geq 1$.  It has a unique expression
\begin{gather*}
a = (x + x') + (y + y') + (z + z') \quad \text{where} \\
x \in \freemult R{E_{\geq 2}G},  y \in \freemult R{F_{\geq 2}G}, z \in \free RP, \\
x' \in \freemult R{F_1E_{\geq 1}G}, y' \in \freemult R{E_1F_{\geq 1}G}, z' \in \freemult R{F_1P}.
\end{gather*}

Notice that one has
\begin{equation*}\del(x + x') \in \freemult R{E_{\geq 1}G}, \quad \del(y + y') \in \freemult R{F_{\geq 1}G}, \quad \del(z + z') \in \free RP.\end{equation*}
Thus, $\del(a) = 0$ implies that each one of $x + x', y + y'$ and $z + z'$ is a cycle.  Next, 
we show each is a boundary,  by giving details for $x + x'$; the other cases are similar.

Since one has $\p  \cap \q  = 0$ in $R$, and $G_{i-1}$ is a free
$R$-module, it follows that $\p (G_{i-1}) \cap \q (G_{i-1}) = 0$.
Therefore, $\del(x + x') = 0$ implies that $x$ and $x'$ are cycles as well.  Let $l(w)$ denote the leftmost letter in the word $w$.
We may express $x$ according to the decomposition
\begin{equation*}\freemult R{E_{\geq 2}G}_i = \bigoplus\limits_{\substack{2 \leq j \leq i \\ w \in \bas G_{i-j} \\ l(w) \in \bas F}} \freemultword R{E_j}{w}.\end{equation*}
If $w$ is basis element of degree $i-j$ with $l(w) \in F$ and $2 \leq j \leq i$,
then one has $\del(\freemultword R{E_j}{w}) \subseteq \freemultword R{E_{j-1}}{w}$.
Hence, each component of $x$ in the decomposition above is a cycle.  For similar reasons the components of $x'$
in $\freemult R{F_1E_{\geq 1}G}$ are cycles.

Therefore, it is enough to show that every cycle of the form
\begin{gather*}
x = \sum\limits_{e \in \bas E_j} r_e ew \in \freemultword R{E_j}{w} \quad\text{or} \\
x' = \sum\limits_{f \in \bas F_1} \sum\limits_{e \in \bas E_{j-1}} r_{fe} few \in \freemultword{R}{F_1E_{j-1}}{w}
\end{gather*}
where $w$ is a fixed word, and $r_e, r_{ef}$ are in $R$, is a boundary.  We give details for $x$, the other case is similar.

We first show that $r_e \in \m$ for each $e \in \bas E_j$.  Indeed, there is a commutative diagram of $R$-modules
\begin{equation} \label{ewstrand} \xymatrix{
\cdots \ar[r] & \freemultword R{E_3}{w} \ar[r] \ar[d]^{\til{\sigma}_3} & \freemultword R{E_2}{w} \ar[r] \ar[d]^{\til{\sigma}_2} &
   \freemultword R{E_1}{w} \ar[r] \ar[d]^{\til{\sigma}_1} & Rw \ar[r] \ar[d]^{\til{\sigma}_0} & 0 \\
\cdots \ar[r] & E_3 \ar[r] & E_2 \ar[r] & E_1 \ar[r] & E_0 \ar[r] & 0
}
\end{equation}
where the vertical maps send $\sum_{e \in \bas E_j} r_e ew$ to $\sum_{e \in \bas E_j} \sigma(r_e) e$.
The image of $x$ in $E$ is a cycle, and hence a boundary, of $E$.  As $E$ is minimal, the claim follows.

Suppose $t = \del(f)$ for $f \in \bas F_1$.  Then $tew = \del(few)$ is a boundary of $G$.  As $F$ is
a resolution of $k$ over $T$, the images of $\bas F_1$ form a minimal generating set for $\q$.  Hence $\q ew$ consists
entirely of boundaries.

The claims above show it suffices to prove the theorem when the coefficients are in $\p$.
In diagram \eqref{ewstrand}, we may also define a morphism of complexes $\til{\gamma} : \p E \to \p \freemultword R{E}{w}$
by sending $\sum_{e \in \bas E_j} s_e e$ to $\sum_{e \in \bas E_j} s_e ew$, viewing $s_e \in \p \subset R$.  Note that
$\til{\gamma}$ and $\til{\sigma}\mid_{\p \freemultword R{E}{w}}$ are inverses of one another.

Suppose that $x \in \p \freemultword R{E}{w}$ is a cycle.  Then $\til{\sigma}(x)$ is a cycle
in $E$.  Hence there exists $u$ so that $\del^E(u) = \til{\sigma}(x)$.  Then one has
\begin{equation*}
\del(\til{\gamma}(u)) = \til{\gamma}(\del^E(u)) = \til{\gamma}(\til{\sigma}(x)) = x.\qedhere
\end{equation*}
\end{proof}

Two special cases of the theorem are used in section \ref{secYonAlg}.

\begin{example}  \label{fieldResln}
When $M = k$, we can take $\bas P = \bas E$.  Let $D$ be the resolution given by Theorem \ref{reslnThm}.  
Since $\bas P_0 = \{1\}$, we can also replace all basis elements of the form $w1$ with a basis element $w$ of the same degree,
and set $\del^{D}(1_R) = 0$.  Therefore, in low degrees, $D$ has the form
\begin{equation*}
\xymatrix@R=-.4pc@C=3.5pc{
& \free RE_3 \ar@<+.7ex>[dr] & & & & \\
&  & \free RE_2 \ar[ddr] & & & \\
& \freemult R{F_1E_2} \ar[ur] & & & & \\
&  & & \free RE_1 \ar[ddddr] & & \\
& \freemult R{E_1F_1E_1} \ar[dr] & & & \\
&  & \freemult R{F_1E_1} \ar[uur] & & & \\
& \freemult R{F_2E_1} \ar[ur] & & & & \\
\ar@{.} '[rrrr] [rrrrr] & & & & R & \\
& \freemult R{E_2F_1} \ar[dr] & & & & \\
&  & \freemult R{E_1F_1} \ar[ddr] & & & \\
& \freemult R{F_1E_1F_1} \ar[ur] & & & & \\
& & & \free RF_1 \ar[uuuur] & & \\
& \freemult R{E_1F_2} \ar[dr] & & & & \\
& & \free RF_2 \ar[uur] & & & \\
& \free RF_3 \ar@<-.5ex>[ur] & & & & 
}
\end{equation*}
\end{example}

\begin{example} \label{ringResln}
When $M = T$, applying the theorem with the roles of $S$ and $T$ reversed, one has
$P_i = 0$ for $i \neq 0$, and $\bas P_0 = \{1\}$.
Let $C$ be the resolution given by Theorem \ref{reslnThm}.  Letting $w$ denote the basis element $w1$ as above, we
see that $C$ is given by the top half of the diagram in Example \ref{fieldResln}.
\end{example}

\section{The Yoneda algebra} \label{secYonAlg}

Let $\extRalg$, $\extSalg$ and $\extTalg$ denote the Ext algebras of
$R$, $S$, and $T$ respectively.  For an $S$-module $M$, we let $\extMS$ be the graded left $\extSalg$-module $\Ext_S(M,k)$,
and let $\extMR$ be the graded left $\extRalg$-module $\Ext_R(M,k)$.

The functor $\Ext_{-}(k,k)$ applied to the diagram of homomorphisms of rings
\begin{equation*}
\xymatrix{R \ar[r]^-{\tau} \ar[d]_-{\sigma} & T \ar[d]^-{\pi_T} \\ S \ar[r]_-{\pi_S} & k} \quad
\substack{ \phantom{\text{\normalsize a}} \\ \text{\normalsize induces a diagram} \\ \text{\normalsize of graded algebras}} \quad
\xymatrix{\extRalg & \extTalg \ar[l]_-{\tau^*} \\ \extSalg \ar[u]^-{\sigma^*} & k \ar[u]_-{\pi_S^*} \ar[l]^-{\pi_T^*}}
\end{equation*}
and hence defines a unique homomorphism of graded $k$-algebras
\begin{equation*}
\phi : \extSalg \amal \extTalg \to \extRalg
\end{equation*}

\begin{theorem} \label{extAlgIso}
The homomorphism of connected $k$-algebras $\phi$ is an isomorphism.
\end{theorem}

The homomorphism $\sigma \colon R \to S$ also induces a homomorphism of graded $\extSalg$-modules
$\sigma^*_M \colon \extMS \to \extMR$.  Since $\sigma^*_M$ is $\sigma^*$-equivariant, the formula
$\xi \otimes \mu \mapsto \xi \cdot \sigma^*_M(\mu)$ defines a homomorphism of graded left $\extRalg$-modules
\begin{equation*}
\theta : \extRalg \otimes_\extSalg \extMS \to \extMR.
\end{equation*}

\begin{theorem} \label{extIso}
The homomorphism of graded left $\extRalg$-modules $\theta$ is an isomorphism.
\end{theorem}

In order to prove Theorems \ref{extAlgIso} and \ref{extIso}, we set up notation and describe the multiplication tables
for $\extRalg$ and $\extMR$.

\begin{notation}
In the notation of Construction \ref{complConstr}, one sees that
\begin{equation*}\extMR \cong \Hom_R(G,k) \cong \Hom_R(G/\m G, k) \cong \Hom_k(G/\m G, k)\end{equation*}
is a $k$-vector space.  Let $\{\dual{w}{R} \mid w \in \bas G\}$ be the graded basis
dual to the image of $\bas G$ in $G/\m(G)$.
Also, let $\{\dual{e}{S} \mid e \in \bas E\}$ be the graded basis dual to the basis given by the image of $\bas E$ in $E/\p E$.
We will abuse language and say that $\dual{w}{R}$ starts (respectively ends) with a letter
from $\bas E$ if the first (respectively last) letter of $w$ is in $\bas E$.  Also, we will say that $\dual{w}{R}$ has length $n$ if
$w$ has length $n$.
\end{notation}

Our first lemma concerns the image of words of length one.

\begin{lemma} \label{extInclLemma}
For $e \in \bas E$, $f \in \bas F$, and $p \in \bas P$, one has
\begin{equation*}
\sigma^*(\dual{e}{S}) = \dual{e}{R}, \quad \tau^*(\dual{f}{S}) = \dual{f}{R} \quad \text{and} \quad \sigma^*_M(\dual{p}{S}) = \dual{p}{R}.
\end{equation*}
\end{lemma}
\begin{proof}
Let $\eps^R \colon D \to k$ and $\eps^S \colon E \to k$ be the augmentation maps.  Set
\begin{equation*}
D' = ~^{R}(\bas D\setminus \bas E) + \q E \subseteq D.
\end{equation*}  The definition of $\del$ shows that $D'$ is a subcomplex.
Also, since $R/\q  \cong S$ one has $D/D' = E$ as complexes of $R$-modules.
Let $\psi$ be canonical surjection $\psi \colon D \to D/D' = E$.  Then one has $\eps^R = \eps^S\psi$, and hence
$\sigma^*(\dual{e}{S}) := \psi\dual{e}{S} = \dual{e}{R}$, as desired.  The other cases are similar.
\end{proof}

Next we provide a partial multiplication table for the action of $\extRalg$ on $\extMR$.

\begin{lemma} \label{multTable}
For $w \in \bas D \cup \bas G$ with starting letter $l(w)$ and $x$ a letter in $\bas E \cup \bas F$, one has
\begin{equation*}
\dual{x}{R}\cdot\dual{w}{R} = \begin{cases}
\dual{fw}{R} & \text{if}~l(w) \in \bas E \cup \bas P~\text{and}~x = f \in \bas F \\
\dual{ew}{R} & \text{if}~l(w) \in \bas F~\text{and}~x = e \in \bas E.
\end{cases}   
\end{equation*}
\end{lemma}
\begin{proof}
Suppose $w = ew' \in \bas G_i$, with $e \in \bas E$.  Let $f$ be an element in $F$.
We define a chain map $\psi_w \in \Hom_R(G,D)_{-i}$ such that $\eps^R\psi_w = \dual{w}{R}$ as follows.  Set
\begin{equation*}\bas G^w = \{x \in \bas G \mid x \not\in (\bas Gw \cup \bas G{\bas E}_{\geq j+1}w')\},\end{equation*}
where $\bas Gw$ denotes elements of $\bas G$ that end in $w$ (including $w$), and $\bas G{\bas E}_{\geq j+1}w'$ denotes
elements of $\bas G$ ending in a letter of $\bas E_{\geq j+1}$, followed by $w'$.
Let $\freemultpower {R}{G}{w}$ be the free $R$-module generated by $\bas G^w$.

The definition of $\del$ shows $\freemultpower {R}{G}{w}$ is a subcomplex of $G$.
Note that $G/\freemultpower {R}{G}{w}$ is a complex of free $R$-modules with basis
$\bas Gw \cup \bas G\bas E_{\geq j + 1}w'$ and differential given by restricting $\del$
to $\freemultword{R}{G}{w}$ and $\freemultword{R}{GE_{\geq j + 1}}{w'}$.
If $v = v'w \in \bas Gw$, define $\alpha'(v) = v'$, and extend $\alpha'$ by $R$-linearity to all of $\freemultword{R}{G}{w}$.
Then for $v = v'w \in \bas Gw$, one has
\begin{equation*}\alpha'(\del(v)) = \alpha'(\del(v'w)) = \alpha'(\del(v')w) = \del(v') = \del(\alpha'(v)).\end{equation*}

Let $B$ be the subcomplex of $G/\freemultpower {R}{G}{w}$ spanned by
$\bas G\bas E_{\geq j+1}w' \cup \{w\}$, and let $C$ be the resolution of $T$ as an $R$-module given
in Example \ref{ringResln}.  As $C$ is acyclic and $B$ is a free complex, we may
define $\alpha'' \colon B \to D$ by lifting the map that sends $w \in B_i$ to $1 \in C_0$ and composing with the
inclusion $C \to D$.  Note that the words in the image of $\alpha''$ end in letters from $\bas E$.
Since $\alpha'(w) = 1_R = \alpha''(w)$, we may define
\begin{equation*}\alpha(v) := \begin{cases} \alpha'(v)  & v \in \freemultword{R}{G}{w} \\
                                            \alpha''(v) & v \in \freemultword{R}{GE_{\geq j+1}}{w'}\end{cases}\end{equation*}
Let $\psi_w$ denote the composition $G \onto G/\freemultpower {R}{G}{w} \xra{\alpha} D$.

Clearly, one has $\eps^R\psi_w = \dual{w}{R}$, hence $\dual{f}{R} \cdot \dual{w}{R} = \dual{f}{R}\psi_w$.
Let $v \in \bas G$ be a word.  If $v \in \bas G^w$, then $\psi_w(v) = 0$.  If $v \in \bas Gw$, then write $v = v'w$.  Then
\begin{equation*}\dual{f}{R}\psi_w(v) = \dual{f}{R}(v') = \begin{cases} 1 & \text{if}~v' = f \\
                                                                        0 & \text{otherwise.}\end{cases}\end{equation*}
If $v \in (\bas G\bas E_{\geq j+1})w'$, then $\psi_w(v)$ is in the span of words with rightmost letters in $\bas E$.
Hence $\dual{f}{R}\psi_w(v) = 0$.
Therefore
\begin{equation*}\dual{f}{R}\cdot\dual{w}{R} = \dual{f}{R}\psi_w(v) = \left\{\begin{array}{cc} 1 & \text{if}~v = fw \\
                                                             0 & \text{otherwise.}\end{array}\right\} = \dual{fw}{R}(v).
\end{equation*}                                                             
The other cases are similar, and often easier.
\end{proof}

\begin{proof}[Proof of Theorem \emph{\ref{extAlgIso}}]
Under the hypothesis of the theorem, the $k$-algebras $\extRalg$, $\extSalg$, and $\extTalg$ are degree-wise finite.
By definition, one has
\begin{equation*}
\hilb{k}{\extRalg} = P^R_k(t), \quad \hilb{k}{\extSalg} = P^S_k(t) \quad \text{and} \quad \hilb{k}{\extTalg} = P^T_k(t).
\end{equation*}
Lemma \ref{hilbLemma} and Corollary \ref{DreKraResult} yield $\hilb{k}{\extRalg} = \hilb{k}{\extSalg \amal \extTalg}$.
As $\phi$ is a homogeneous $k$-linear map, it suffices to show that it is surjective.  Lemma \ref{multTable} shows that
\begin{equation*}\{\dual{e}{R} \mid e \in \bas E\} \cup \{\dual{f}{R} \mid f \in \bas F\}\end{equation*}
generates $\extRalg$ as a $k$-algebra.  Lemma \ref{extInclLemma} shows that these generators are in the image of $\phi$.
\end{proof}

\begin{proof}[Proof of Theorem \emph{\ref{extIso}}]
By Lemma \ref{hilbLemma} and Corollary \ref{DreKraResult} the Hilbert series of $\extRalg \otimes_{\extSalg} \extMS$ and $\extMR$
are equal, so it is enough to show that $\theta$ is surjective. By Lemma \ref{multTable}, $\extMR$ is generated as a
left $\extRalg$-module by $\{\dual{p}{R} \mid p \in \bas P\}$.  By
Lemma \ref{extInclLemma}, $\theta(1 \otimes \dual{p}{S}) = \dual{p}{R}$, and hence $\theta$ is surjective.
\end{proof}

Recall that a graded module $M$ over a connected algebra $A$ is said to be \emph{Koszul} when $\Ext^i_A(M,k)_j = 0$ if $j \neq i$.
A connected algebra $A$ is Koszul if $k$ is Koszul as an $A$-module.  By Remark \ref{gradedRemark}, the homomorphisms $\phi$ and
$\theta$ preserve the internal gradings of $\extSalg \amal \extTalg$ and $\extRalg$, giving the next corollary.  The
equivalence of the first two conditions was proved in \cite{BacFro85}.

\begin{corollary} \label{kosBacFro}
The following conditions are equivalent.
\begin{enumerate}
\item The algebra $R$ is Koszul.
\item The algebras $S$ and $T$ are Koszul.
\item There exists an $S$-module $M$ that is Koszul as a $R$-module. \hfill \qed
\end{enumerate}
\end{corollary}

The functor $\Ext_R(-,k)$ has a property similar to the one given in Theorem \ref{extAlgIso}.

\begin{theorem} \label{thmExtFibMod}
Let $M$, $N$ and $V$ be $S$, $T$ and $k$-modules respectively, such that there
exist surjective $\pi_S$ and $\pi_T$-equivariant homomorphisms $M \xra{\mu} V \xla{\nu} N$
with $\ker \mu = \p M$ and $\ker \nu = \q N$.  The exact sequence of $R$-modules
\begin{equation*}0 \to M \times_V N \xra{\iota} M \times N \xra{\mu - \nu} V \to 0\end{equation*}
induces an exact sequence of graded left $\extRalg$-modules
\begin{equation*}0 \to \extRalg \otimes_k V^* \xra{(\mu^*,-\nu^*)} \extMR \times \extNR \xra{\iota^*} \extLR \to 0\end{equation*}
where $\extLR = \Ext_R(M \times_V N,k)$, and $V^* = \Hom_k(V,k)$.  In particular, there is an equality of Poincar\'e series
\begin{equation*}P_{M\times_V N}^R(t) + (\rank_k V)P^R_k(t) = P^R_M(t) + P^R_N(t).\end{equation*}
\end{theorem}

\begin{proof}
The sequence of $R$-modules defining $M \times_V N$ induces an exact sequence of graded $\extRalg$-modules
\begin{equation*}
\shift^{-1} \extLR \to \extRalg \otimes_k V^* \xra{(\mu^*,-\nu^*)} \extMR \times \extNR \xra{\iota^*} \extLR \to \shift(\extRalg \otimes_k V^*).
\end{equation*}

Thus, we need to show that $(\mu^*,-\nu^*)$ is injective. Set $n = \rank_k V$.
One has $S$-linear maps $S^n \to M \xra{\mu} V$ that induce homomorphisms of graded $\extRalg$-modules
$\extSalg \otimes_k V^* \to \extMS \to k^n$.  Tensoring with $\extRalg$ over $\extSalg$ on the left, one obtains
\begin{equation*}
\extRalg \otimes V^* \xra{\mu^*} \extMR \to \extRalg \otimes_\extSalg k^n
\end{equation*}
since $\extMR \cong \extRalg \otimes_\extSalg \extMS$.  Under the isomorphism in Theorem \ref{extAlgIso}, the kernel of this
composition, and hence $\mu^*$, is in the span of elements the elements of $\extRalg \otimes_k V^*$ ending in $\bas E_{\geq 1}$.
Similarly, one can show that the kernel of $\nu^*$ is contained the span of those elements of $\extRalg \otimes_k V^*$ ending in
$\bas F_{\geq 1}$.  Hence $\Ker (\mu^*,-\nu^*) = \Ker \mu^* \cap \Ker \nu^* = 0$.
\end{proof}

\section{Depth of cohomology modules} \label{secDepth}

The notation and conventions given in sections \ref{secComplexConstr} and \ref{secYonAlg} are still in force.

In order to describe the cohomology module for an
arbitrary $R$-module $L$, we use an observation of Dress and Kr\"amer in \cite[Remark 3]{DreKra75}.
Recall that the syzygy $\syz^R_1 L$ of an $R$-module 
$L$ is the kernel of a free cover $F \to L$; it is defined uniquely up to isomorphism; for $n \geq 2$ one sets
$\syz_n^R L = \syz^R_1 \syz_{n-1}^R L$.
\begin{proposition} \label{DreKraSyz}
Let $L$ be a left $R$-module. Then $\syz_2^R(L) \cong M \oplus N$ where $M$ is an $S$-module and $N$ is a $T$-module.
\end{proposition}

\begin{proof}
Recall that the maximal ideal of $R$ is $\m = \p \oplus \q$.   Let $\varphi : A \to B$ be a minimal free presentation 
of $L$ over $R$.  One then has
\begin{eqnarray*}
\syz_2^R(L) = \Ker \varphi & = & \Ker \varphi \cap \m A \\
& = & \Ker \varphi \cap (\p A \oplus \q A) \\
& = & (\Ker \varphi \cap \p A) \oplus (\Ker \varphi \cap \q A)
\end{eqnarray*}
To see the last equality, suppose that $(x_1,x_2)$ in $\p A \oplus \q A$ satisfies
$\varphi((x_1,x_2)) = 0$.  Note that $\varphi((x_1,0))$ is in $\varphi(\p A) \subseteq \p B$ and
$\varphi((0,x_2))$ is in $\varphi(\q A) \subseteq \q B$.  Also,
$\p B \cap \q B = 0$, hence $\varphi((x_1,0)) = 0 = \varphi((0,x_2))$.  Taking
$M = \Ker \varphi \cap \p A$ and $N = \Ker \varphi \cap \q A$
gives the desired result.
\end{proof}

By putting together Theorem \ref{extIso} and Proposition \ref{DreKraSyz}, we obtain a nearly
complete description of the cohomology of arbitrary $R$-modules.

\begin{corollary} \label{calRModExSeq}
Set $\extLR = \Ext_R(L,k)$, $\extMS = \Ext_S(M,k)$, and $\extNT = \Ext_T(N,k)$.
There is then an exact sequence of graded left $\extRalg$-modules
\begin{equation*}
0 \to (\shift^{-2}\extRalg \otimes_\extSalg \extMS)\oplus(\shift^{-2}\extRalg \otimes_\extTalg \extNT)
\to \extLR \to \extLR/\extLR^{\geq 2} \to 0 
\end{equation*}
\end{corollary}

Corollary \ref{calRModExSeq} allows us to compute the depth of the cohomology
module of an $R$-module.  For uses of this invariant, see \cite{FHJLT88} or \cite{AvrVel05}.

\begin{definition}
For a graded connected $k$-algebra $\calE$ and a graded left $\calE$-module $\calM$,
one defines the depth of $\calM$ over $\calE$ by means of the formula
\begin{equation*}\depth_\calE \calM = \inf \{ n \in \bbN \mid \Ext_\calE^n(k,\calM) \neq 0 \}.\end{equation*}
\end{definition}

For $M = k$, the following theorem recovers \cite[Example 36.(e).2]{FHT01}.  The idea to use the
resolution of $\extSalg \amal \extTalg$ constructed from those of $\extSalg$ and $\extTalg$ is taken
from there.  Note that for each $i$, $\Ext^i_\extRalg(k,\extMR)$ is a graded abelian group, with the internal
grading given by the cohomological grading on $\extMR$.  We will denote its $j$th graded piece by
$\Ext^i_\extRalg(k,\extMR)^j$.

More precisely, if $\calF_i$ is the $i^\text{\scriptsize th}$ free module in a graded free resolution of $k$ over
$\extRalg$, then $\Ext_\extRalg^i(k, \extMR)$ carries the induced grading of $\Hom_\extRalg(\calF_i, \extMR)$.
A homomorphism $\varphi \colon \calF_i \to \extMR$ has degree $j$ if it satisfies
$\varphi(\calF_i^n) \subseteq \extMR^{n-j}$ for all $n \in \mathbb{Z}$.

\begin{theorem} \label{depthSMod}
If $M$ is a finitely generated non-zero $S$-module, and neither $\pi_S$ nor $\pi_T$ is an isomorphism.
Then one has $\depth \extMR = 1$.

More precisely, for each $j \geq 2$, one has
\begin{equation*}
\Ext_\extRalg^1(k,\extMR)^j \oplus \Ext_\extRalg^1(k,\extMR)^{j+1} \neq 0,
\end{equation*}
unless $S$ and $T$ both have global dimension $1$ and $M$ is free, in which case
\begin{equation*}
\Ext_\extRalg^1(k,\extMR)^1 \neq 0.
\end{equation*}
\end{theorem}
\begin{proof}
As $S$ and $T$ are not fields, one can find $\varsigma \in \extSalg^1 \setminus \{0\}$ and $\vartheta \in \extTalg^1 \setminus \{0\}$.
We fix these elements for the remainder of the proof.  We first show that for each $\mu \in \extMR^i$, there is an element
$\xi \in \extRalg^+$ so that $\xi\mu \neq 0$.  By Theorem \ref{extIso} and \ref{freeProdMod}, we may arrange
the terms in $\mu$ so that $\mu = \alpha + \beta + \gamma$, where the terms in $\alpha$ start
with a letter from $\bas E$, the terms in $\beta$ start with $\bas F$ and the terms in $\gamma$ start with $1$.
If $\gamma \neq 0$, then the terms in $\vartheta\gamma$ start with a letter of $\bas F_1$, and the length of
each of its terms is 2.  But $\vartheta\alpha$ has terms of length greater than or equal to 3, and $\vartheta\beta$ has terms with
the degree of their leading letters greater than 1.  Therefore, no term in $\vartheta\mu$ can cancel
$\vartheta\gamma$, so $\vartheta\mu \neq 0$.  One can similarly argue that if $\alpha$ or $\beta$
are nonzero, then $\vartheta\alpha$ or $\varsigma\beta$ are nonzero, respectively.

Choose free resolutions of $k$ over $\extSalg$ and $\extTalg$ and write them in the form
\begin{equation*}
\xymatrix@R=.5pc@C=1pc{
\cdots \ar[r]^-{\del} & \extSalg \otimes_k V(2) \ar[r]^-{\del} & \extSalg \otimes_k V(1) \ar[r]^-{\del} & \extSalg \\
\cdots \ar[r]^-{\del} & \extTalg \otimes_k W(2) \ar[r]^-{\del} & \extTalg \otimes_k W(1) \ar[r]^-{\del} & \extTalg 
}
\end{equation*}
with graded $k$-vector spaces $V(i)$ and $W(i)$.  One then has a free resolution
\begin{equation} \label{freeResExtR}
\xymatrix@C=1pc{
\cdots \ar[r]^-{\del} & \extRalg \otimes_k (V(2) \oplus W(2)) \ar[r]^-{\del} & \extRalg \otimes_k (V(1) \oplus W(1)) \ar[r]^-{\del} & \extRalg 
}\end{equation}
of $k$ over $\extRalg$, where the differentials are given by restricting those in the resolutions of $k$ over $\extSalg$ and $\extTalg$, see
\cite[Example 36.e.2]{FHT01}.

For any pair of elements $\alpha,\beta \in \extMR^i$, define an $\extRalg$-linear map
\begin{equation*}
\begin{gathered}
\phi_{\alpha,\beta} \colon \extRalg \otimes_k (V(1) \oplus W(1)) \to \extMR \\
(v,w) \mapsto (\del(v)\alpha,\del(w)\beta)
\end{gathered}
\end{equation*}
Then $\phi_{\alpha,\beta}$ is a 1-cocycle and hence defines a class in $\Ext_\extRalg(k,\extMR)$.  If there
exists $v \in V(1)$ or $w \in W(1)$ so that $(\del v)\beta \neq 0$ or $(\del w)\alpha \neq 0$, then $\phi_{\alpha,\beta}$
represents a nonzero cohomology class.
 
As $M$ is nonzero, there exists $\mu \in \extMS^0 \setminus \{0\}$.  If $M$ is not free over $S$, then there
also exists $\mu' \in \extMS^1 \setminus \{0\}$.  If $\gldim S \geq 2$ (respectively $\gldim T \geq 2)$, then there exists
$\varsigma' \in \extSalg^2 \setminus \{0\}$ (respectively $\vartheta' \in \extTalg^2 \setminus \{0\}$).  For each $j \geq 1$, define

\begin{equation*}
\begin{gathered}
\alpha_j = \begin{cases} (\vartheta\varsigma)^{j-1}\vartheta\mu     & \text{if $M$ is not free} \\
                         (\vartheta\varsigma)^{j-1}\vartheta'\mu    & \text{if}~\gldim T \geq 2 \\
                         (\vartheta\varsigma)^j\vartheta\mu         & \text{if}~\gldim S \geq 2 \end{cases} \\
\beta_i = \begin{cases} (\varsigma\vartheta)^{j-1}\mu'                     & \text{if $M$ is not free} \\
                        (\varsigma\vartheta)^j\mu                          & \text{if}~\gldim T \geq 2 \\
                        (\varsigma\vartheta)^{j-1}\varsigma'\vartheta\mu   & \text{if}~\gldim S \geq 2 \end{cases}
\end{gathered}
\end{equation*}

Then for each $j \geq 1$, $\phi_{\alpha_j,\beta_j}$ defines a distinct nonzero class of $\Ext^1_\extRalg(k,\extMR)$,
with an internal degree of $\phi_{\alpha_j,\beta_j}$ either $-2j+1$, $-2j$, or $-2j-1$ if $M$ is not free,
$\gldim T \geq 2$ or $\gldim S \geq 2$, respectively.

When $S$ and $T$ have global dimension one, the graded free resolution \eqref{freeResExtR} is
\begin{equation*}
\xymatrix@C=1pc{
0 \ar[r] & \extRalg \otimes_k (V(1) \oplus W(1)) \ar[r]^-{\del} & \extRalg 
}\end{equation*}
Application of $\Hom_\extRalg(-,\extMR)$ gives a short exact sequence:
\begin{equation*}
0 \to \Hom_\extRalg(\extRalg,\extMR) \xra{\del^*} \Hom_\extRalg(\extRalg \otimes_k (V(1) \oplus W(1)), \extMR) \to
  \Ext^1_\extRalg(k,\extMR) \to 0
\end{equation*}
Since $\extSalg$ and $\extTalg$ are connected $k$-algebras, the graded bases $V(1)$ and $W(1)$ start in degree one,
and hence $\Hom_\extRalg(\extRalg \otimes_k (V(1) \oplus W(1)), \extMR)^1 \neq 0$.  Furthermore, this component is
not in the image of $\del^*$ since $\Hom_\extRalg(\extRalg,\extMR)$ starts in degree zero and $\del^*$ is homogeneous.
Hence $\Ext^1_\extRalg(k,\extMR)^1 \neq 0$.
\end{proof}

Using the short exact sequence given in Corollary \ref{calRModExSeq}, one also has the following:

\begin{corollary}
For each non-zero finitely generated $R$-module $L$, one has
\begin{equation*}
\depth_\extRalg \extLR \leq 1.
\end{equation*}
\end{corollary}
\begin{proof}
If $\pd_R L$ is finite, then $\rank_k \extLR$ is finite, hence $\depth_\extRalg \extLR = 0$.  If $\pd_R L = \infty$, then
$\syz^2_R(L) \neq 0$, and by Proposition \ref{DreKraSyz}, we have $\syz^2_R(L) = M \oplus N$ for some $S$-module $M$ and some
$T$-module $N$, with $M$ or $N$ non-zero.

Suppose that either $\gldim S \geq 2$ or $\gldim T \geq 2$, and set $\calX := \Ext_R(M\oplus N, k)$.
Then $\rank_k \Ext^1_\extRalg(k,\calX) = \infty$, by Theorem \ref{depthSMod}.  Then by Corollary \ref{calRModExSeq},
together with Theorem \ref{depthSMod}, there is an exact sequence
\begin{equation*}
\Hom_\extRalg(k, \extLR / \extLR^{\geq 2}) \xra{\varphi} \Ext^1_\extRalg(k,\calX) \xra{\psi}
\Ext^1_\extRalg(k,\extLR).\end{equation*}
Since $\rank_k \Hom_\extRalg(k,\extLR / \extLR^{\geq 2})$ is finite, $\varphi$ is not surjective, and hence $\psi$ is nonzero.

If both $S$ and $T$ have global dimension one, then one has $\Ext_\extRalg^1(k,\calX)^1 \neq 0$ by
Theorem \ref{depthSMod}.  Since $\varphi$ is a homogeneous homomorphism and $\Hom_\extRalg(k,\extLR / \extLR^{\geq 2})$
is concentrated in nonnegative degrees, $\varphi$ is not surjective, and hence $\psi$ is nonzero.
\end{proof}

\section*{Acknowledgements}

The author would like to thank Luchezar Avramov, Greg Peipmeyer, Srikanth Iyengar and Anders Frankild for many useful
discussions and comments throughout the evolution of this paper.

\providecommand{\bysame}{\leavevmode\hbox to3em{\hrulefill}\thinspace}
\providecommand{\MR}{\relax\ifhmode\unskip\space\fi MR }
\providecommand{\MRhref}[2]{%
  \href{http://www.ams.org/mathscinet-getitem?mr=#1}{#2}
}
\providecommand{\href}[2]{#2}

\end{document}